\documentclass[12pt]{amsart}
\usepackage{SSdefn}
\usepackage{enumitem}
\usepackage{tikz-cd}
\usepackage{tikz}
\usepackage{fdsymbol}
\usepackage{mathtools}
\usepackage[font=footnotesize]{caption}
\usepackage{graphicx} 
\usepackage{quiver}
\tikzcdset{column sep/normal=0.3cm}
\title{Non-noetherian $\GL$-algebras in characteristic two}
\author{Karthik Ganapathy}
\address{Department of Mathematics, University of California, San Diego, CA}
\email{\href{mailto:karthg@umich.edu}{karthg@umich.edu}}
\thanks{}
\urladdr{\url{https://public.websites.umich.edu/~karthg/}}
\usetikzlibrary{%
  matrix,%
  calc,%
  arrows%
}

\theoremstyle{plain}

\begin{document}
\begin{abstract}
Over fields of characteristic two, we construct an infinite ascending chain of $\GL$-stable ideals in the coordinate ring of infinite skew-symmetric matrices. This provides the first known example of a non-noetherian $\GL$-algebra resolving a long-standing open question in the area. Our results build on the work of Draisma--Krasilnikov--Krone.
\end{abstract}
\maketitle
\section{Introduction}\label{s:intro}
In his study of the Segre embedding, Snowden \cite{sno13delta} proved a noetherian property for coordinate independent ideals in infinite-variable polynomial rings. Similar results can be found in, or deduced from, many independent streams of work spanning multiple decades \cite{co67laws,adf80wedge2,ah07sym,dra10fin,cef15fi}. Snowden recognized these disparate results as part of a larger theory---that of $\GL$-algebras, which are commutative $k$-algebras with a \textit{polynomial} action of the infinite general linear group $\GL \coloneqq \GL(k^{\infty})$. 
Since then, specific $\GL$-algebras have been comprehensively analyzed \cite{ss16gl,ss19gl2, gan22ext, ganglp} with the noetherian property being established only in specific instances \cite{ss17grobner, nss16deg2, nss19skew}. 
A key breakthrough came with Draisma's proof \cite{dra19top} that all finitely generated $\GL$-algebras are \textit{topologically $\GL$-noetherian}, meaning the descending chain condition holds for \textit{$\GL$-stable closed subsets} of their spectrum. Draisma's theorem can be used to prove and even generalize Stillman's conjecture \cite{ess19big, dll19still, ess19gen} (see also \cite{ess19cub}). This rich history suggests an affirmative answer to the algebraic noetherianity problem, often described as the ``most tantalizing problem in this area" \cite[Page 709]{dra19top}:
\begin{question}
\label{ques:main}
    Let $P$ be a finite length polynomial representation of $\GL$. Does every ascending chain of $\GL$-stable ideals $I_1 \subset I_2 \subset \ldots \subset I_n \subset \ldots$ in $k[P]$ eventually stabilize?
\end{question}
The goal of this paper is to negatively answer this question when $\chark(k) = 2$ which we assume moving forward. Let $S = \Sym(\lw^2(k^\infty))$ and $R = \lw(\lw^2(k^\infty))$; we identify $S$ with the polynomial algebra $k[x_{i, j} \vert i,j \in \bN;  i<j]$ and since $\chark(k)=2$, we identify $R$ with the quotient of $S$ by the $\GL$-stable ideal $\fm^{[2]} = (x_{i, j}^{2})$. 

For each $n \geq 2$, let $I_n$ be the (ordinary) ideal in $R$ generated by the Pl\"ucker elements and monomials of ``length $\leq n$ cycles"; see \S\ref{s:preliminaries} for exact definitions. Our key insight is:
\begin{theorem}\label{thm:glstable}
The ideal $I_n$ is stable under the action of $\GL$ for all $n \geq 2$. 
\end{theorem}
Combining this with Draisma--Krasilnikov--Krone's theorem \cite{dkkinc} that every inclusion in the chain $I_2 \subset I_3 \subset \ldots \subset I_n \subset \ldots $ is strict (Theorem~\ref{thm:dkk}), we immediately obtain:
\begin{theorem}\label{thm:main}
When $P = \lw^2(k^\infty)$ and $\chark(k) = 2$, the answer to Question~\ref{ques:main} is no.
\end{theorem}
We emphasize that in char zero, even $\GL$-equivariant $S$-modules and $R$-modules satisfy the noetherian property \cite{nss16deg2,nss18reg}. Topological $\GL$-noetherianity for $\Spec(S)$ follows by classical linear algebra in all characteristics. Theorem~\ref{thm:main} is all the more surprising as the algebra $R/I_2$ is even $\fS_{\infty}$-noetherian in char zero \cite[Proposition~4.9]{dra10fin}. 

Theorem~\ref{thm:main} implies the non-noetherianity of $k[P]$ whenever $P$ contains a subquotient of the form $\lw^2(k^\infty)$ (and more generally, $\lw^2(k^\infty) \otimes Q$). This includes the cases $P = \Sym^2(k^{\infty})$, $\Div^2(k^{\infty})$ and $(k^{\infty})^{\otimes 2}$. 
Our counterexample does \textit{not} establish the non-noetherianity of the the corresponding twisted commutative algebra ``FIM" \cite[Remark~1.3]{nss16deg2} (see also \cite{mw19higher}) in char two as the generators of $I_{n+1}/I_n$ are ``Frobenius twisted" $\GL$-representations which vanish on the tca side (see \cite[\S2]{gan22ext}).

We conclude by further illustrating the subtlety of Theorem~\ref{thm:main}. Recall that a $\GL$-stable ideal $I$ is \textit{$\GL$-prime} if given $\GL$-stable ideals $\fa$ and $\fb$ with $\fa \fb \subset I$, either $\fa \subset I$ or $\fb \subset I$. The $\GL$-spectrum $\Spec_{\GL}$ is the set of $\GL$-prime ideals endowed with the Zariski topology. 
\begin{theorem}\label{thm:glspec}
For all $n \geq 2$, the ideals $I_n$ define the same closed subset of $\Spec_{\GL}(R)$.
\end{theorem}
In char $0$, Snowden \cite{sno20spe} proved that the $\GL$-spectrum is noetherian for a finitely generated $\GL$-algebra. The above result suggests that Snowden's result may still hold in char $p$.

\S\ref{s:proofs} contains all proofs of our results. Our proofs use fairly elementary methods. In some sense, our key insight was to work with the exterior algebra where the relative ease of explicit computations led us to guess and subsequently prove Theorems~\ref{thm:dkk} and \ref{thm:glstable} (see also \cite[\S1.2.3]{gan22ext}).

\subsection*{Acknowledgements} I am indebted to Andrew Snowden for carefully reading a previous draft and providing numerous comments which greatly improved the exposition. I also thank Jan Draisma for sharing \cite{dkkinc} with me.

\section{Recollections}\label{s:preliminaries}
We endow the algebras $S$ and $R$ with an $\bN^{\infty}$-grading by setting $\deg(x_{i, j}) = \epsilon_i + \epsilon_j$ where $\epsilon_r$ is $1$ in the $r$-th coordinate and $0$ elsewhere. If we write $x_{i,j}$ with $i > j$, we are abusing notation and mean $x_{j,i}$. 

There is a useful way to depict elements of $R$ which we now recall. Each monomial $m$ in $R$ corresponds to an undirected simple graph $G_m$ with vertex set $\bN$ and edge set $E(G_m) = \{ \{i,j\} \colon  x_{i,j} \text{ divides } m \}$. For two monomials $m$ and $m'$, we have
\[
E(G_{m m'})= 
\begin{cases}
    E(G_m) \cup E(G_{m'}) & \text{if }E(G_m) \cap E(G_{m'}) = \emptyset \\
    \emptyset & \text{if }E(G_m) \cap E(G_{m'}) \ne \emptyset. \\
\end{cases}
\]
We often omit isolated vertices in our graphs.
\begin{defn}
Let $I_2$ be the ideal of $R$ generated by the  Pl\"ucker elements 
\[\pl_{i_1,i_2,i_3,i_4} = x_{i_1,i_2}x_{i_3,i_4}+ x_{i_1,i_3}x_{i_2,i_4}+x_{i_1,i_4}x_{i_2,i_3}.\] For $n > 2$, let $I_n$ be the ideal generated by the Pl\"ucker elements and monomials $m$ where $G_m$ is an $i$-cycle with $2 < i \leq n$.
\end{defn}

Each ideal $I_n$ is homogeneous with respect to the $\bN^{\infty}$-grading on $R$. For $n>2$, let $w_n = x_{1,2}x_{2,3} \cdots x_{n-1,n}x_{n,1}$ be the ``standard $n$-cycle" and $\pl = x_{1,2}x_{3,4}+x_{1,3}x_{2,4}+x_{1,4}x_{2,3}$ be the ``standard Pl\"ucker element"; all $n$-cycles and Pl\"ucker elements are a permutation of $w_n$ and $\pl$ respectively.

The next result is essentially a result of Draisma--Krasilnikov--Krone \cite{dkkinc}; the proof here is a streamlined version of their proof.
\begin{theorem}[cf.~Theorem~43 in \cite{dkkinc}]\label{thm:dkk}
The monomial $w_{n+1}$ is not in $I_n$ for all $n \geq 2$.
\end{theorem}
\begin{proof}
   Let $\bd$ denote the multidegree $(2,2,2, \ldots, 2, 0, 0, \ldots )$ in total degree $2(n+1)$. The monomials in degree $\bd$ correspond to graphs with $n+1$ edges on $\{1, 2, \ldots, 2n+1, 2n+2\}$ whose connected components are a union of cycles. We decompose the $k$-vector space $R_{\bd}$ as
 \[R_{\bd} \cong \cL_{1} \oplus \cL_{2} \oplus \cL_{3}  \]
 where $\cL_{1}$ is spanned by monomials $m$ with $G_m$ containing an $i$-cycle with $3 \leq i< n+1$, $\cL_{2}$ is the subspace \{$\sum_{\sigma \in \fS_{n+1}} a_{\sigma} \sigma(w_{n+1}) \vert \sum_{\sigma} a_{\sigma} = 0$\}, and $\cL_{3} = k\{w_{n+1}\}$. 
 
 For the sake of contradiction, assume $w_{n+1}\in I_n$. Then we may write $w_{n+1} = a + b$ with $a \in \cL_{1}$ and $b = \sum_{i=1}^r c_i m_i \sigma_i(\pl) \in R_{\bd}$ with $c_i \in k^{\times}$, $m_i$ a monomial, $\sigma_i \in \fS_{n+1}$ and $m_i \sigma_i(\pl) \in R_{\bd}$ for all $0 < i \leq r$. We may further assume that each $m_i \notin I_{n}$ as we can subsume such terms into $a$. Since $\pl$ has degree $(1, 1, 1, 1)$, we get that $n_i \coloneqq \sigma_i^{-1}(m_i)$ has multidegree $(1,1,1,1,2,2, \ldots, 2)$. The graph $G_{n_i}$ does not contain any cycles (as $m_i \notin \cL_1$) so 
 $G_{n_i}$ will be a disjoint union of two paths with endpoints $\{1,2,3,4\}$. Expanding $\pl$, we get $n_i \pl = n_i x_{1,2}x_{3,4} + n_i x_{1,3}x_{2,4}+ n_i x_{1,4}x_{2,3}$ which is represented graphically below.
 Without loss of generality {(since $\chark(k)=2$)}, we may assume the path in $G_{n_i}$ that starts at $1$ ends at $2$ (and so the other path has endpoints $3$ and $4$). Therefore, the first term in the expansion lies in $\cL_{1}$ and the sum of the other two terms lies in $\cL_{2}$ as $1+1=0$ in $k$.
\[ n_i\pl = 
\begin{tikzcd}[cramped, row sep=tiny, column sep=small]
	2 & 4 & &{}& & 2 & 4 & &{} & & 2 & 4 \\
	1 & 3 & &{}& & 1 & 3 & &{} & & 1 & 3
	\arrow[no head, from=1-11, to=1-12]
	\arrow[curve={height=-30pt}, dotted, no head, from=2-1, to=1-1]
	\arrow[no head, from=2-1, to=1-1]
	\arrow[no head, from=2-2, to=1-2]
	\arrow[""{name=0, anchor=center, inner sep=0}, curve={height=30pt}, dotted, no head, from=2-2, to=1-2]
	\arrow[""{name=1, anchor=center, inner sep=0}, curve={height=-30pt}, dotted, no head, from=2-6, to=1-6]
	\arrow[no head, from=2-6, to=1-7]
	\arrow[no head, from=2-7, to=1-6]
	\arrow[""{name=2, anchor=center, inner sep=0}, curve={height=30pt}, dotted, no head, from=2-7, to=1-7]
	\arrow[""{name=3, anchor=center, inner sep=0}, curve={height=-30pt}, dotted, no head, from=2-11, to=1-11]
	\arrow[no head, from=2-11, to=2-12]
	\arrow[curve={height=30pt}, dotted, no head, from=2-12, to=1-12]
	\arrow["{+}"{description}, draw=none, from=1-4, to=2-4]
	\arrow["{+}"{description}, draw=none, from=1-9, to=2-9]
\end{tikzcd}\] 
So for each $i$,  the element $m_i \sigma_i(\pl) = \sigma_i(n_i \pl) \in \cL_{1} \oplus \cL_{2}$ and so $w_{n+1} = a + \sum_i{c_i m_i \sigma_i(\pl)} \in \cL_{1} \oplus \cL_{2}$ contradicting the fact that $w_{n+1} \notin \cL_{1} \oplus \cL_{2}$.
\end{proof}
\begin{remark}\label{rmk:hmsv}
The above result is somewhat complementary to a key trick used to study the ideal of relations of the GIT quotient $(\mathbb{P}^1)^n // \SL_2$ in \cite{hmsv12moduli}. Lemma~6.3 in loc.~cit.~is about decomposing elements corresponding to $n$-cycles into smaller ones using the Pl\"ucker relations over $\bZ[\frac{1}{2}]$ whereas we show above that in char $2$, larger cycles cannot be obtained from smaller ones using the Pl{\"u}cker relation. Over fields of char $\ne 2$, the element $w_n$ is even in the $\GL$-stable ideal generated by just $w_{3}$. 
\end{remark}
\section{Proofs of Theorems~\ref{thm:glstable},~\ref{thm:main}, and~\ref{thm:glspec}}\label{s:proofs}
\begin{proof}[Proof of Theorem~\ref{thm:glstable}]
We proceed by induction on $n$; the base case of $n =  2$ is well known as the ideal of Pl\"ucker relations is $\GL$-stable. Let $E_{2,1} \in \GL$ be the matrix with $E_{2,1}(e_1) = e_1 + e_2$ and $E_{2,1}(e_i) = e_i$ for all $i \ne 1$. The group $\GL$ is generated by the element $E_{2,1}$, the permutation matrices, and the diagonal matrices. It suffices to show that for each generator $g$ of $\GL$ and each generator $m$ of $I_n$, we have $g(m) \in I_n$. This is clearly true if $g$ is a permutation matrix or a diagonal matrix. We now treat the case where $g = E_{2,1}$.

Using the induction hypothesis, we may assume that $m$ corresponds to the $n$-cycle on $i_1, i_2, \ldots, i_n$. If $1 \notin \{i_1,i_2, \ldots, i_n\}$, then $E_{2,1}(m) = m$. So assume $1 \in \{i_1, i_2, \ldots, i_n\}$ and without loss of generality, let $i_1 = 1$. Writing $m = m' x_{i_n,1}x_{1, i_2}$, we have $E_{2,1}(m) = m'(x_{i_n,1}x_{1,i_2}+x_{i_n,2}x_{2,i_2} + x_{i_n,2}x_{1,i_2}+x_{i_n,1}x_{1,i_2})$ which we represent graphically below [note that $2$ could coincide with one of the $i_j$].
\[
\begin{tikzcd}[cramped, row sep=tiny]
	& {i_j} & {i_{n-1}} &&& {i_j} & {i_{n-1}} &&& {i_j} & {i_{n-1}} &&& {i_j} & {i_{n-1}} \\
	{i_3} && 2 & {i_n} & {i_3} && 2 & {i_n} & {i_3} && 2 & {i_n} & {i_3} && 2 & {i_n} \\
	& {i_2} & 1 &&& {i_2} & 1 &&& {i_2} & 1 &&& {i_2} & 1
	\arrow[dotted, no head, from=1-2, to=1-3]
	\arrow[no head, from=1-3, to=2-4]
	\arrow[dotted, no head, from=1-6, to=1-7]
	\arrow[no head, from=1-7, to=2-8]
	\arrow[dotted, no head, from=1-10, to=1-11]
	\arrow[no head, from=1-11, to=2-12]
	\arrow[dotted, no head, from=1-14, to=1-15]
	\arrow[no head, from=1-15, to=2-16]
	\arrow[dotted, no head, from=2-1, to=1-2]
	\arrow[no head, from=2-4, to=3-3]
	\arrow[dotted, no head, from=2-5, to=1-6]
	\arrow["{+}"{description}, draw=none, from=2-5, to=2-4]
	\arrow[no head, from=2-7, to=3-6]
	\arrow[no head, from=2-8, to=2-7]
	\arrow[dotted, no head, from=2-9, to=1-10]
	\arrow["{+}"{description}, draw=none, from=2-9, to=2-8]
	\arrow[no head, from=2-12, to=2-11]
	\arrow["{+}"{description}, draw=none, from=2-12, to=2-13]
	\arrow[dotted, no head, from=2-13, to=1-14]
	\arrow[no head, from=2-13, to=3-14]
	\arrow[no head, from=2-15, to=3-14]
	\arrow[no head, from=2-16, to=3-15]
	\arrow[no head, from=3-2, to=2-1]
	\arrow[no head, from=3-3, to=3-2]
	\arrow[no head, from=3-6, to=2-5]
	\arrow[no head, from=3-10, to=2-9]
	\arrow[no head, from=3-11, to=3-10]
\end{tikzcd}\]

Consider the degenerate case where $i_j = 2$ for some $j \leq n$. If $2< j < n-1$, then each of the individual terms in $E_{2,1}(m)$ are in $I_n$ since they all contain an $r$-cycle with $r \leq n$. If $j = 2$ (i.e., $i_2=2$), then the second and fourth terms are zero, and the first term is an $n$-cycle and the third term contains an $n-1$-cycle. The case $j=n$ (i.e., $i_n = 2$) is similar. So $E_{2,1}(m) \in I_{n}$ if $i_j = 2$ for some $j$.

If $i_j \ne 2$ for all $j$, then all four terms are nonzero. Clearly the first and second term are in $I_n$ being permutations of $w_n$. Using the Pl\"ucker relation $\pl$, we rewrite the last two terms as
\[\begin{tikzcd}[cramped, row sep=tiny]
	& {i_j} & {i_{n-1}} &&& {i_j} & {i_{n-1}} &&& {i_j} & {i_{j-1}} \\
	{i_3} && 2 & {i_n} & {i_3} && 2 & {i_n} & {i_3} && 2 & {i_n} \\
	& {i_2} & 1 &&& {i_2} & 1 &&& {i_2} & 1
	\arrow[dotted, no head, from=1-2, to=1-3]
	\arrow[no head, from=1-3, to=2-4]
	\arrow[dotted, no head, from=1-6, to=1-7]
	\arrow[no head, from=1-7, to=2-8]
	\arrow[dotted, no head, from=1-10, to=1-11]
	\arrow[no head, from=1-11, to=2-12]
	\arrow[dotted, no head, from=2-1, to=1-2]
	\arrow[no head, from=2-4, to=2-3]
	\arrow["{+}"{description}, draw=none, from=2-4, to=2-5]
	\arrow[dotted, no head, from=2-5, to=1-6]
	\arrow[no head, from=2-5, to=3-6]
	\arrow[no head, from=2-7, to=3-6]
	\arrow["{=}"{description}, draw=none, from=2-8, to=2-9]
	\arrow[no head, from=2-8, to=3-7]
	\arrow[dotted, no head, from=2-9, to=1-10]
	\arrow[dashed, no head, from=2-12, to=3-10]
	\arrow[no head, from=3-2, to=2-1]
	\arrow[no head, from=3-3, to=3-2]
	\arrow[no head, from=3-10, to=2-9]
	\arrow[no head, from=3-11, to=2-11]
\end{tikzcd}\]
which also lies in $I_{n}$ being a multiple of a permutation of $w_{n-1}$, as required.
\end{proof}
\begin{proof}[Proof of Theorem~\ref{thm:main}]
The chain of ideals $I_2 \subset I_3 \subset I_4 \ldots \subset I_n \subset \ldots$ is $\GL$-stable by Theorem~\ref{thm:glstable} and does not stabilize since $w_{n+1} \in I_{n+1}\setminus I_n$ by Theorem~\ref{thm:dkk}. Therefore, $R$ is not $\GL$-noetherian, and since $S \twoheadrightarrow R$, neither is $S$.
\end{proof}
\begin{remark}[Odd characteristics]\label{rmk:otherchar}
Our example does not generalize to odd characteristics (see Remark~\ref{rmk:hmsv}) but examples from \cite{dkk13sym} may lead to infinite ascending chains in $\Sym(\lw^p(k^{\infty}))$ when char$(k) = p >2$. We provide a word of caution however. Draisma--Krasilnikov--Krone's \cite{dkkinc} construction of the chain $\{I_n\}_{n\in \bN}$ was inspired by Vaughan-Lee's \cite{vl75lie} example of a \textit{variety} of abelian-by-nilpotent Lie algebras in char $2$ which is not {finitely based} (i.e., not defined by finitely many equations); such varieties \textit{are} finitely based when char $\ne 2$ \cite{bv72sol}.
\end{remark} 
We finally prove Theorem~\ref{thm:glspec} after some preparatory results. Given $f_1, f_2 \in R$, we write $f_1 \leadsto_I f_2$ if $f_2$ is in the $\GL$-stable ideal generated by $f_1$ in $R/I$.

\begin{lemma}\label{lem:trivalent}
    Let $m = m' x_{1,4}x_{2,4}x_{3,4}$ with $m'$ a monomial in $R$ such that $x_{4,i}\nmid m'$ for all $i$. Then $m \leadsto_0 n \coloneqq m'(x_{1,5}x_{2,4}x_{3,4} + x_{1,4}x_{2,5}x_{3,4}+ x_{1,4}x_{2,4}x_{3,5})$.
\end{lemma}
\begin{proof}
Applying the Lie algebra element $e_{4,5}$ to $m$, we get $e_{4,5}(m) = e_{4,5}(m') x_{1,4}x_{2,4}x_{3,4} + m' e_{4,5}(x_{1,4}x_{2,4}x_{3,4})$. Since $x_{4,i} \nmid m'$, we have $e_{4,5}(m') = m'$ and using Leibniz rule, we get $e_{4,5}(x_{1,4}x_{2,4}x_{3,4}) = x_{1,5}x_{2,4}x_{3,4} + x_{1,4}x_{2,5}x_{3,4} + x_{1,4}x_{2,4}x_{3,5}$. So $m \leadsto_0 e_{4,5}(m) = m+n$ or, $m \leadsto_0 n$.
\end{proof}
The content of Lemma~\ref{lem:trivalent} is that if the graph $G_m$ of a monomial $m$ contains a trivalent vertex $l$, the $\GL$-ideal it generates contains the trinomial obtained by replacing each edge $(a, l)$ with $(a, N)$ one by one:
\[\begin{tikzcd}[cramped, sep=tiny]
	& k &&& k &&&& k &&&& k \\
	j & l & \leadsto_0 & j & l & N && j & l & N && j & l & N \\
	& i &&& i &&&& i &&&& i
	\arrow[no head, from=1-5, to=2-6]
	\arrow[no head, from=1-9, to=2-9]
	\arrow[no head, from=2-1, to=2-2]
	\arrow[no head, from=2-2, to=1-2]
	\arrow[no head, from=2-2, to=3-2]
	\arrow[no head, from=2-4, to=2-5]
	\arrow[no head, from=2-5, to=3-5]
	\arrow["{+}"{description}, draw=none, from=2-6, to=2-8]
	\arrow[no head, from=2-8, to=2-9]
	\arrow["{+}"{description}, draw=none, from=2-10, to=2-12]
	\arrow[shift right=1, curve={height=-10pt}, dashed, no head, from=2-12, to=2-14]
	\arrow[no head, from=2-13, to=1-13]
	\arrow[no head, from=2-13, to=3-13]
	\arrow[no head, from=3-9, to=2-10]
    \end{tikzcd}\]

\begin{proposition}\label{prop:mIn+1}
    For all $n \geq 2$, we have $\fm I_{n+1} \subset I_n$.
\end{proposition}
The key combinatorial idea in the proof is that (modulo $I_n$) we can turn an $n$-cycle appended with a ``tail" of length two to an $n+1$-cycle with a disjoint edge, using the $\GL$-action and the Pl\"ucker relations. The monomial corresponding to the first graph is in $I_n$ and the monomial corresponding to the second graph \textit{generates} $\fm I_{n+1}$ so $\fm I_{n+1} \subset I_n$. 
\begin{proof}[Proof of Proposition~\ref{prop:mIn+1}]
We skip the $n=2$ case as it can be computed explicitly and instead focus on the $n > 2$ case. It suffices to prove $w_{n+1} x_{n+2,n+3} \in I_n$ since the $\GL$-ideal it generates is $\fm I_{n+1}$ by \cite[Corollary~2.17]{ganglp}.
   
Let $\sigma$ be a permutation with $1 \mapsto 3$ and $i \to i+3$ for $1 < i \leq n+1$. Consider the element $x_{2,4}x_{3,4}\sigma(w_n)$; it's graph is the $n$-cycle on $3, 5, 6, \ldots, n+3$ with the ``tail" $3-4-2$. We proceed to show $x_{2,4}x_{3,4}\sigma(w_n) \leadsto_{I_n} w_{n+1} x_{n+2,n+3}$. We first use $\pl_{2,3,4,5}$ to rewrite this monomial into a binomial, the second term of which clearly vanishes as $x_{3,4}^2 = 0$ in $R$.
\[\begin{tikzcd}[cramped, row sep = small]
& {n+3} &&& {n+3} && {n+3} \\
5 & 3 &\leadsto_{I_2}& 5 & 3 & 5 & 3 \\
4 & 2 && 4 & 2 & 4 & 3 
\arrow[no head, from=1-7, to=2-7]
\arrow[curve={height=-12pt}, dotted, no head, from=2-1, to=1-2]
\arrow[no head, from=2-1, to=2-2]
\arrow[no head, from=2-2, to=1-2]
\arrow[no head, from=2-2, to=3-1]
\arrow[curve={height=-12pt}, dotted, no head, from=2-4, to=1-5]
\arrow[no head, from=2-4, to=3-4]
\arrow[no head, from=2-5, to=1-5]
\arrow["{+}"{description}, draw=none, from=2-5, to=2-6]
\arrow[no head, from=2-5, to=3-5]
\arrow[curve={height=-12pt}, dotted, no head, from=2-6, to=1-7]
\arrow[no head, dashed, from=2-6, to=3-7]
\arrow[no head, from=3-1, to=3-2]
\arrow[no head, from=3-4, to=2-5]
\arrow[Rightarrow, no head, from=3-6, to=2-7]
\end{tikzcd}\] 
Next, we apply Lemma~\ref{lem:trivalent} to the trivalent vertex $3$ in the graph of the first term to obtain a trinomial.
\[\begin{tikzcd}[cramped, row sep=small]
	& {n+3} &&& {n+3} &&& {n+3} &&& {n+3} &&& {n+3} \\
	5 & 3 & {\leadsto_{I_2}} & 5 & 3 & {\leadsto_0} & 5 & 3 & 1 & 5 & 3 & 1 & 5 & 3 & 1 \\
	4 & 2 && 4 & 2 && 4 & 2 && 4 & 2 && 4 & 2 \\
	\arrow[no head, from=1-2, to=2-2]
	\arrow[no head, from=1-8, to=2-9]
	\arrow[no head, from=1-14, to=2-14]
	\arrow[curve={height=-12pt}, dotted, no head, from=2-1, to=1-2]
	\arrow[no head, from=2-2, to=2-1]
	\arrow[no head, from=2-2, to=3-1]
	\arrow[curve={height=-12pt}, dotted, no head, from=2-4, to=1-5]
	\arrow[no head, from=2-4, to=3-4]
	\arrow[no head, from=2-5, to=1-5]
	\arrow[no head, from=2-5, to=3-4]
	\arrow[no head, from=2-5, to=3-5]
	\arrow[curve={height=-12pt}, dotted, no head, from=2-7, to=1-8]
	\arrow[no head, from=2-7, to=3-7]
	\arrow[no head, from=2-8, to=3-8]
	\arrow["{{+}}"{description}, draw=none, from=2-9, to=2-10]
	\arrow[curve={height=-12pt}, dotted, no head, from=2-10, to=1-11]
	\arrow[no head, from=2-11, to=1-11]
	\arrow[no head, from=2-11, to=3-11]
	\arrow["{{+}}"{description}, draw=none, from=2-12, to=2-13]
	\arrow[curve={height=-12pt}, dotted, no head, from=2-13, to=1-14]
	\arrow[no head, from=2-13, to=3-13]
	\arrow[no head, from=3-1, to=3-2]
	\arrow[no head, from=3-7, to=2-8]
	\arrow[no head, from=3-10, to=2-10]
	\arrow[dashed, no head, from=3-10, to=2-12]
	\arrow[no head, from=3-13, to=2-14]
	\arrow[no head, from=3-14, to=2-15]
\end{tikzcd}\]
The first two terms in the resulting trinomial can be combined using $\pl_{1,3,4,n+3}$ into a monomial, which lies in $I_n$ as it's graph contains the $n$-cycle $4,5,6,\ldots, n+3$.
\[\begin{tikzcd}[cramped, row sep=small]
	& {n+3} &&& {n+3} &&& {n+3} \\
	5 & 3 & {\leadsto_{I_2}} & 5 & 3 & 1 & 5 & 3 & 1 \\
	4 & 2 && 4 & 2 && 4 & 2
	\arrow[curve={height=-12pt}, dotted, no head, from=2-1, to=1-2]
	\arrow[no head, from=2-1, to=2-2]
	\arrow[no head, from=2-2, to=1-2]
	\arrow[no head, from=2-2, to=3-1]
	\arrow[no head, from=3-1, to=3-2]
	\arrow[curve={height=-12pt}, dotted, no head, from=2-4, to=1-5]
	\arrow[no head, from=2-4, to=3-4]
	\arrow[no head, from=2-5, to=2-6]
	\arrow[no head, from=2-5, to=3-5]
	\arrow["{+}"{description}, draw=none, from=2-6, to=2-7]
	\arrow[curve={height=-12pt}, dotted, no head, from=2-7, to=1-8]
	\arrow[no head, from=2-7, to=3-7]
	\arrow[no head, from=2-8, to=1-8]
	\arrow[no head, from=3-4, to=1-5]
	\arrow[no head, from=3-7, to=2-8]
	\arrow[no head, from=3-8, to=2-9]
\end{tikzcd}\]
So $x_{2,4}x_{3,4} \sigma(w_n) \leadsto_{I_n} x_{1,2} (x_{3,4}x_{4,5} \ldots x_{n+2,n+3}x_{3,n+3})$ which is a permutation of $w_{n+1} x_{n+2,n+3}$. Since $w_n \in I_n$, we have $w_{n+1} x_{n+2, n+3} \in I_n$ as required.
\end{proof}
\begin{proof}[Proof of Theorem~\ref{thm:glspec}]
For $n \geq 2$, we have the inclusion $I_n \subset I_{n+1}$ by definition, and by Proposition~\ref{prop:mIn+1}, we also have $I_{n+1}^2 \subset I_n$ so they define the same closed subset in the Zariski topology on $\Spec_{\GL}(R)$.
\end{proof}
\begin{remark}
   Consider the $\GL$-equivariant map $\phi \colon R \to \lw[x_1, x_2,x_3 \ldots, y_1, y_2,y_3\ldots]$ where $x_{i, j} \mapsto x_iy_j + x_j y_i$. The analogous map in the polynomial ring case is well-studied: it's kernel is generated by the Pl\"ucker relations. In this case, we easily see that $I_2 \subsetneq \ker(\phi) = I_{\infty} \coloneqq \cup I_n$. Furthermore, since $(0)$ is a $\GL$-prime in the target of $\phi$, we get that $I_{\infty}$ is a $\GL$-prime ideal in $R$.
\end{remark}
 
\bibliographystyle{abbrv}
\bibliography{bibliography}
\end{document}